\documentclass[12pt,reqno]{amsart}    
\usepackage{setspace}
\usepackage{amsmath,amssymb}
\usepackage{amsthm}
\usepackage[lite]{amsrefs}
\usepackage{amsfonts}
\usepackage{url}
\usepackage{enumitem}
\usepackage{graphicx}
\usepackage{color}
\usepackage{epstopdf}
\usepackage{tikz-cd}
\usepackage{pgfplots} 
\pgfplotsset{compat=newest}
\usetikzlibrary{decorations.pathreplacing,angles,quotes}
\usetikzlibrary{bending}
\usepackage{mathtools}
\usepackage[colorlinks]{hyperref}
\usepackage[all]{hypcap}
\usepackage{esint}

\numberwithin{equation}{section}

\DefineSimpleKey{bib}{myurl}
\newcommand\myurl[1]{\url{#1}}

\BibSpec{webpage}{%
	+{}{\PrintAuthors} {author}
	+{,}{ \textit} {title}
	+{}{ \parenthesize} {date}
	+{,}{ } {journal}
	+{,}{ } {note}
	+{}{ \myurl} {myurl}
	+{.}{ } {transition}
}

\newcommand{\sub}{\subseteq}

\newcommand{\R}{\mathbb{R}}
\newcommand{\C}{\mathbb{C}}

\newcommand{\N}{\mathbb{N}}

\newcommand{\bs}{\backslash}

\newcommand{\Ga}{\alpha}
\newcommand{\Gb}{\beta}
\newcommand{\Gd}{\delta}
\newcommand{\Ge}{\varepsilon}

\newcommand{\Gl}{\lambda}
\newcommand{\Go}{\omega}

\newcommand{\Gs}{\sigma}

\numberwithin{chap}{section}
\newtheorem{thm}{Theorem}
\numberwithin{thm}{section}
\newtheorem{prop}[thm]{Proposition}
\newtheorem{defn}[thm]{Definition}
\newtheorem{lem}[thm]{Lemma}
\newtheorem{cor}[thm]{Corollary}

\DeclarePairedDelimiter{\norm}{\lVert}{\rVert}
\DeclarePairedDelimiter{\abs}{\lvert}{\rvert}
\DeclarePairedDelimiter{\ip}{\langle}{\rangle}

\makeatletter
\let\oldabs\abs
\def\abs{\@ifstar{\oldabs}{\oldabs*}}
\let\oldnorm\norm
\def\norm{\@ifstar{\oldnorm}{\oldnorm*}}
\let\oldip\ip
\def\ip{\@ifstar{\oldip}{\oldip*}}
\makeatother

\begin{document}           

\pagestyle{myheadings} \thispagestyle{empty} \markright{}

\title{Uniform $l^2$-decoupling in $\mathbb R^2$ for Polynomials}

\author{Tongou Yang}

\address{Department of Mathematics, University of British Columbia, Vancouver, B.C. Canada V6T 1Z2}
\email{toyang@math.ubc.ca}

\maketitle

\begin{abstract}
   For each positive integer $d$, we prove a uniform $l^2$-decoupling inequality for the collection of all polynomials phases of degree at most $d$. Our result is intimately related to \cite{MR4078083}, but we use a different partition that is determined by the geometry of each individual phase function.
\end{abstract}


\section{Introduction}

\subsection{Background}
Bourgain and Demeter's breakthrough $l^2$-decoupling theorem \cite{MR3374964} in 2015, based on previous founding work by \cites{MR2288738,MR1800068} and many others, has led to substantial applications in harmonic analysis and number theory. Its original proof has also been studied in further details, leading to improvements of the original result, such as \cite{Li}. Moreover, based on the techniques they developed, a lot of analogues and generalisations of the classical decoupling theorem have been also well studied, most notably \cite{MR3548534} which connects harmonic analysis to number theory. Similar results include, but not limited to, \cites{MR4078083,MR3736493,Guo,Kemp,MR3939285}.

In \cite{MR4078083}, the authors studied the $l^2$-decoupling inequality on $\R^2$ for general real analytic phase functions other than the unit parabola. They proved that for all $2\leq p\leq 6$ (the same range of exponents for the classical decoupling inequality in \cite{MR3374964}), and for every real analytic phase function $\phi$ with $\phi''$ vanishing to a finite order $r$ on $[0,1]$, the extension operator version of the (local) $l^2$-decoupling inequality holds with the partition of $[0,1]$ by intervals of the same length $\Gd^{1/2}$, but with the smallest physical square replaced by a long thin tube of dimensions $\Gd^{-1}$ and $\Gd^{-1-r/2}$.

Our result will be closely related to \cite{MR4078083}. For each $d\geq 1$, we will prove a uniform $l^2$-decoupling inequality with the phase function belonging to the family of all polynomials of degree at most $d$. Our result will be significantly different from \cite{MR4078083} in the following aspects. First and most importantly, instead of the canonical partition by intervals of length $\Gd^{1/2}$, we will consider partitions that are closely related to the geometry of individual phase functions, which plays a key role in the uniformity of the decoupling inequalities we will prove. A similar formulation can also be found in Section 12.6 of Demeter's textbook \cite{MR3971577}, where he used such technique to prove a decoupling inequality for a single analytic phase function.

Second, in terms of technicality, we will be dealing with neighbourhood versions of the decoupling inequality without physical localisation. In the appendix, we show that our result actually implies Proposition 2.1 of \cite{MR4078083} in the case when $\nu\geq 0$. Nevertheless, one notable common feature of our results is that we both use the main result of \cite{MR3374964} as a foundation of our argument.

\subsection{Motivation}
To familiarise the reader of the motivation of this paper, we start with the classical Bourgain-Demeter decoupling of the unit parabola $\phi(s)=s^2$, $s\in [0,1]$. Let $\Gd\in 2^{-2\N}$ and partition $[0,1]$ into intervals $I$ of length $\Gd^{1/2}$. Then the $\Gd$-neighbourhood of the graph of $\phi$ over each $I$ is essentially a rectangle of side length $\Gd^{1/2}\times \Gd$ with the longer side pointing in the tangent line of $\phi$ at $(c_I,\phi(c_I))$ where $c_I$ is the centre of $I$. This feature is essentially due to the non-vanishing curvature of $s^2$ on $[0,1]$.

However, if we consider $\phi(s)=s^3$, this is no longer the case when $I$ is close to the origin. Instead of using the canonical partition, we will construct another partition that suits the geometry of the curve better. Let $a_0=0$. Starting from $a_0$, we find the largest number $a_1$ such that the $\Gd$-neighbourhood of the tangent line at $a_0$ contains the graph of $\phi$ over $[a_0,a_1]$. In this case, we see that $a_1=\Gd^{1/3}$. We then start from $a_1$ and find the largest number $a_2$ such that the $\Gd$-neighbourhood of the tangent line at $a_1$ contains the graph of $\phi$ over $[a_1,a_2]$. Proceeding in this way, we obtain a sequence $a_j$ which is essentially $a_j=j^{2/3}\Gd^{1/3}$, $0\leq j\leq \Gd^{-1/2}$ if we take $\Gd\in 2^{-6\N}$. This partition of $[0,1]$ by such $a_j$ will be an ``admissible partition" for $\phi(s)=s^3$ at scale $\Gd$, which we formally define in a moment. Thus, given a family of functions $f_j$ where each $f_j$ is Fourier supported on the $\Gd$-neighbourhood of the graph of $s^3$ over $[(j-1)^{2/3}\Gd^{1/3},j^{2/3}\Gd^{1/3}]$, we ask if the following inequality holds for some range of $p$:
\begin{equation}\label{s^3}
    \norm {\sum_{j=1}^{\Gd^{-1/2}} f_j}_{L^p(\R^2)}\lesssim_{\Ge}\Gd^{-\Ge}\left(\sum_{j=1}^{\Gd^{-1/2}}\norm {f_j}^2_{L^p(\R^2)}\right)^{\frac 1 2},
\end{equation}
where the implicit constant depends on $p$ and $\Ge$ but not $\delta$ or $f_j$. As we shall see, if $2\leq p\leq 6$, then this inequality will follow as a particular case of our main result.

\subsection{Admissible partitions}
We now formally define our notion of admissible partitions. Roughly speaking, an admissible partition $\mathcal P$ for a phase function $\phi$ at scale $\Gd$ is the coarsest partition such that over each $I\in \mathcal P$, the graph of $\phi$ differs by $O(\Gd)$ from any tangent line of $\phi$ over $I$. Equivalently, $\mathcal P$ is the coarsest partition such that for each $I\in \mathcal P$, the vertical neighbourhood of thickness $\Gd$ of the graph of $\phi$ above $I$ is an approximate rectangle.

Now we start to define this notion. Let $I_0$ be a compact interval and $\phi:I_0\to \R$ be $C^2$, and let $I\sub I_0$ be a subinterval. For $r>0$, we say $I$ satisfies property $P(r)$ if
$$
\sup_{s,c\in I}|\phi(s)-\phi(c)-\phi'(c)(s-c)|\leq 2r.
$$
By Taylor expansion, the above is true if and only if 
\begin{equation}\label{Taylor}
  \sup_{s\in I}|\phi''(s)||I|^2\leq 4r.  
\end{equation}

\begin{defn}[Admissible Partition]\label{admispart}
Let $I_0$ be a compact interval and $\phi:I_0\to \R$ be $C^2$. Let $\mathcal P$ be a finite partition of $I_0$ into subintervals $I$. We say $\mathcal P$ is 
\begin{enumerate}
\item a super-admissible partition of $I_0$ for $\phi$ at scale $r$, if each $I\in \mathcal P$ satisfies property $P(r)$.
\item a sub-admissible partition of $I_0$ for $\phi$ at scale $r$, if for each pair of adjacent intervals $I,J\in \mathcal P$, their union $I\cup J$ does not satisfy property $P(r)$.
\item an admissible partition of $I_0$ for $\phi$ at scale $r$, if it is both super-admissible and sub-admissible for $\phi$ at scale $r$.
\end{enumerate}
\end{defn}
It follows immediately from the definition that $\{[(j-1)\Gd^{1/2},j\Gd^{1/2}]\}_{j=1}^{\Gd^{-1/2}}$ is an admissible partition for $s^2$ at scale $\Gd\in 2^{-2\N}$ and $\{[(j-1)^{2/3}\Gd^{1/3},j^{2/3}\Gd^{1/3}]\}_{j=1}^{\Gd^{-1/2}}$ is an admissible partition for $s^3$ at scale $\Gd\in 2^{-6\N}$. Also, in view of \eqref{Taylor}, if $\norm {\phi''}_\infty\leq 4\Gd$, then all sub-admissible partitions at scale $\Gd$ become the trivial partition $\{[0,1]\}$.

More properties of admissible partitions will be discussed in Section \ref{sectionadmis}.

\subsection{The main theorem}
To state the main theorem, we need the following standard notation.
\begin{defn}[Vertical neighbourhood]
Let $I_0\sub \R$ be an interval and let $\phi:I_0\to \R$ be any function. For any interval $I\sub I_0$ and $r>0$, let $\mathcal N^\phi_{I,r}$ denote the following curved vertical neighbourhood of the graph of $\phi$ over $I$:
$$
\mathcal N^\phi_{I,r}=\{(s,t):s\in I,|t-\phi(s)|\leq r\}.
$$
\end{defn}
\begin{defn}[Fourier truncation]\label{fourierres}
Let $1<p<\infty$. For any $f\in L^p(\R^2)$ with $\hat f\in L^1(\R^2)$ and any interval $I$, we define $f_I$ to be such that 
$$
\hat f_{I}(s,t)=\hat f (s,t)1_{I}(s).
$$
\end{defn}
{\it Remark.} By the boundedness of the Hilbert transform, the function $(s,t)\mapsto 1_I(s)$ defines a bounded Fourier multiplier on $L^p(\R^2)$ for all $1<p<\infty$, with the operator norm depending on $p$ but not $I$. In particular, we still have $f_I\in L^p(\R^2)$ for any $1<p<\infty$.

The following is one of our main theorems.
\begin{thm}\label{poly}
For any $2\leq p\leq 6$, $d\geq 1$ and $\Ge>0$, there is a constant $C_\Ge=C_{d,\Ge,p}$ such that the following is true. For any $0<\Gd\leq 1$, any polynomial $\phi$ of degree at most $d$, any sub-admissible partition $\mathcal P$ of $[0,1]$ for $\phi$ at scale $\Gd$ and any $f\in L^p(\R^2)$ Fourier supported on $\mathcal N^{\phi}_{[0,1],\Gd}$, we have
\begin{equation}\label{mainpoly}
    \norm f_{L^p(\R^2)}\leq  C_\Ge (\Gd^{-1}\sup_{s\in [0,1]}|\phi''(s)|)^\Ge\left(\sum_{I\in \mathcal P}\norm {f_I}^2_{L^p(\R^2)}\right)^{\frac 1 2}.
\end{equation}
In particular, if we choose $\phi$ such that all of its coefficients are bounded by some constant independent of $\Gd$, then the constant in \eqref{mainpoly} can be taken to be uniform in such polynomials $\phi$.
\end{thm}
 
\subsection{Polynomial-like functions}
Theorem \ref{poly} is highly reliant on some special feature of the collection of polynomials of degree $d$, which we formally state as follows.

\begin{defn}[Polynomial-like functions]\label{polylike}
For $d\geq 1$, let $C_d>1$ be a constant depending on $d$ only. We define $\mathcal D_d=\mathcal D_d(C_d)$ to be the collection of all functions $\phi\in C^3([0,1])$ obeying the following condition: for each $0<\Gs<C_d^{-d}$ and any interval $J\sub [0,1]$, the set
\begin{equation}\label{defB}
    B=B(\phi,J):=\{s\in J:|\phi''(s)|<\Gs(\sup_{s\in J}|\phi''(s)|+|J|\sup_{s\in J}|\phi'''(s)|)\}
\end{equation}
is a union of at most $C_d$ disjoint subintervals relatively open in $J$ and also satisfies $|B|\leq C_d\Gs^{1/d}|J|$.
\end{defn}

The previous main theorem \ref{poly} follows as an immediate corollary of the following slightly more general theorem and Lemma \ref{lemboot}.
\begin{thm}[Uniform $l^2$-decoupling for functions in $\mathcal D_d$]\label{uniform}
For any $2\leq p\leq 6$, $d\geq 1$, $C_d>1$ and $\Ge>0$, there is a constant $C_\Ge=C_{p,d,C_d,\Ge}$ such that the following is true. For any $0<\Gd\leq 1$, any $\phi\in \mathcal D_d=\mathcal D_d(C_d)$, any sub-admissible partition $\mathcal P$ of $[0,1]$ for $\phi$ at scale $\Gd$ and any $f\in L^p(\R^2)$ Fourier supported on $\mathcal N^{\phi}_{[0,1],\Gd}$, we have
\begin{equation}\label{uniformdecoupling}
    \norm f_{L^p(\R^2)}\leq  C_\Ge \norm{\phi''}^\Ge_\infty\Gd^{-\Ge}\left(\sum_{I\in \mathcal P}\norm {f_I}^2_{L^p(\R^2)}\right)^{\frac 1 2}.
\end{equation}
If, in addition, $\norm {\phi''}_\infty$ can be taken to be bounded by some constant independent of $\Gd$, then \eqref{uniformdecoupling} holds with the constant independent of the choice of $\phi$ in such subcollection.
\end{thm}

\begin{lem}\label{lemboot}
If $d\geq 3$, then there is some $C_{d-2}$ such that any polynomial $\phi$ of degree at most $d$ belongs to $\mathcal D_{d-2}=\mathcal D_{d-2}(C_{d-2})$.
\end{lem}

The family $\mathcal D_d$ contains functions other than polynomials of degree at most $d+2$.
For example, with a suitable choice of $C_d$, all functions $\phi(s)=s^\gamma$ where $\gamma\in [3,d+2]$ are in $\mathcal D_d(C_d)$. Also, all functions of the form $\sin (\Go s+\Go_0)$ and $\cos(\Go s+\Go_0)$ where $|\Go|\leq 1$ are contained in $\mathcal D_1(C_1)$ for some suitable $C_1$. Finally, it is easily observed that if $\phi\in \mathcal D_d$, then for all $a,b,c\in \R$, the phase $s\mapsto c\phi(s)+as+b\in \mathcal D_d$.

{\it Remark.} \begin{enumerate}
    \item Theorem \ref{uniform} immediately leads to \eqref{s^3}.
    \item Theorem \ref{uniform} works for all sub-admissible partitions, although it is the sharpest for admissible partitions. In particular, this shows that the decoupling inequality for the unit parabola is true if we replace the canonical partition $[j\Gd^{1/2},(j+1)\Gd^{1/2}]$ by any coarser partition (without any assumption on the structure of the partition).
\end{enumerate}
{\it Outline of the paper.} In Section \ref{secproofmain} we give a proof of Theorem \ref{uniform} assuming a key bootstrap inequality \eqref{mainboot}, and we also prove Lemma \ref{lemboot}. In Section \ref{sectionadmis} we explore more properties of admissible partitions which will be needed subsequently. In Section \ref{secnonzerocurve} we study decoupling inequalities for curves with nonvanishing curvature, which contains a uniform decoupling result for all quadratic polynomials. In Section \ref{secpolyres} we discuss polynomial rescaling, which will in turn be used in Section \ref{bootproof} to prove the bootstrap inequality \eqref{mainboot}. In the appendix, we show that our result implies Proposition 2.1 of \cite{MR4078083}.

{\it Notation.}
\begin{enumerate}
    \item As always, the notation $A\lesssim_{\Ga,\Gb}B$ means that there is a constant $C$, depending on $\Ga$ and $\Gb$, such that $A\leq CB$.
    \item We write $e(z)=e^{2\pi i z}$ for all $z\in \C$. This is the only place that $i$ appears as the imaginary unit in this paper, so that we can save this notation as an index. 
\end{enumerate}

{\it Acknowledgement.} I sincerely thank my advisor Malabika Pramanik for her kind guidance in the preparation of this paper. I also thank Zane Li for useful discussions, especially about the technical aspects of the world of decoupling.

\section{Proof of main theorem}\label{secproofmain}
The proof of Theorem \ref{uniform} is by induction on scales. We need a little more terminology.

For $d\geq 1$, we fix a choice of $C_d$ which is large enough. Let $\mathcal A_d=\mathcal A_d(C_d)$ be the subcollection of $\phi\in\mathcal D_d=\mathcal D_d(C_d)$ with $\norm {\phi''}_\infty$ bounded above by $1$.

\begin{defn}[Uniform decoupling for functions in $\mathcal A_d$]\label{defdecAd}
Let $1<p<\infty$. Let $D^d_p(\Gd)$, $\Gd>0$ be the best constant such that for each $\phi\in \mathcal A_d$, each sub-admissible partition $\mathcal P$ of $[0,1]$ for $\phi$ at scale $\Gd$, each function $f\in L^p(\R^2)$ with Fourier support in $\mathcal N^\phi_{[0,1],\Gd}$, we have
$$
\norm {f}_{L^p(\R^2)}\leq D^d_p(\Gd)\left(\sum_{I\in\mathcal P}\norm {f_I }^2_{L^p(\R^2)}\right)^{\frac 1 2}.
$$
\end{defn}
{\it Remark.} We have $1\leq D^d_p(\Gd)<\infty$, where the finiteness follows from Proposition \ref{number} in Section \ref{sectionadmis}, the fact that $\sup_{\phi\in \mathcal A_d}\norm {\phi''}_\infty\leq 1$ and an application of triangle inequality and Cauchy-Schwarz. Also, $D^d_p(\Gd)$ implicitly depends on the choice of absolute constant $C_d$.

The key bootstrap inequality we shall use the following.
\begin{thm}\label{mainbootthm}
Let $2\leq p\leq 6$, $\Ge>0$, $d\geq 1$. If $C_d$ is large enough, then for any $M>C_d$, there is a constant $C_{\Ge,M}=C(p,d,\Ge,M)$ and a constant $K=K(d)$, such that for each $0<\Gd\leq 1$ we have
\begin{equation}\label{mainboot}
    D^d_p(\Gd)\leq K(C_{\Ge,M}\Gd^{-\Ge}+ \sup_{\Gd'\geq M\Gd}D^d_p(\Gd')).
\end{equation}
\end{thm}

With this, we prove
\begin{thm}\label{familyA}
Let $2\leq p\leq 6$, $\Ge>0$, $d\geq 1$. Then there is $C_\Ge=C_{\Ge,p,d}$ such that for all $0<\Gd\le 1$, we have
$$
D^d_p(\Gd)\leq C_\Ge \Gd^{-\Ge}.
$$
\end{thm}

\begin{proof} 
Let $M>C_d$ be determined at the end of the proof. Denote $S(r):=\sup_{\Gd'\geq r}D^d_p(\Gd')$. Then \eqref{mainboot} says that
\begin{equation}\label{mainboot2}
  D^d_p(\Gd)\leq K(C_{\Ge,M} \Gd^{-\Ge}+S(M\Gd)).  
\end{equation}
For any $\Gd'\geq M\Gd$ but $\Gd'\leq 1$, we may apply \eqref{mainboot} again to get
$$
D^d_p(\Gd')\leq K(C_{\Ge,M} \Gd'^{-\Ge}+S(M\Gd')).
$$
Taking supremum over $\Gd'\geq M\Gd$ to the above equation, we have
$$
S(M\Gd)\leq K(C_{\Ge,M} M^{-\Ge}\Gd^{-\Ge}+S(M^2\Gd))\leq K(C_{\Ge,M}\Gd^{-\Ge}+S(M^2\Gd)).
$$
Plugging into \eqref{mainboot2}, we have
$$
D^d_p(\Gd)\leq KC_{\Ge,M} \Gd^{-\Ge}+K^2 C_{\Ge,M} \Gd^{-\Ge}+K^2S(M^2\Gd).
$$
In general, for each $n\geq 0$, if $M^n\Gd\leq 1$, then
\begin{align*}
    D^d_p(\Gd)
    &\leq K^n S(M^n \Gd)+\sum_{m=1}^n C_{\Ge,M} \Gd^{-\Ge}K^m\leq K^n S(M^n \Gd)+nK^nC_{\Ge,M} \Gd^{-\Ge}.
\end{align*}
We choose $n\geq 0$ to be the smallest integer such that $M^n \Gd\geq 1$, and thus $n<\frac{\log(\Gd^{-1})}{\log M}+1$.

We also have $S(M^n \Gd)=1$. Indeed, if $\Gd'\geq M^n\Gd\geq 1$, then all sub-admissible partitions of $[0,1]$ for $\phi$ at scale $\Gd'$ must be the trivial partition, in view of \eqref{Taylor} and the assumption that $\norm {\phi''}_\infty \leq 1$.

Thus
\begin{align*}
     K^n S(M^n \Gd)+nK^nC_{\Ge,M} \Gd^{-\Ge}
    \leq K^n +(2K)^nC_{\Ge,M} \Gd^{-\Ge}
    \leq (3K)^n C_{\Ge,M} \Gd^{-\Ge}.
\end{align*}
and thus
\begin{align*}
   D^d_p(\Gd)
   &\leq (3K)^{\frac {\log (\Gd^{-1})}{\log M}+1} C_{\Ge,M}\Gd^{-\Ge}\\
   &=(\Gd^{-1})^{\frac {\log (3K)}{\log M}}3KC_{\Ge,M}\Gd^{-\Ge}.
\end{align*}
Now we may choose $M=(3K)^{1/\Ge}>C_d$, so $M$ depends on $d,\Ge$ only. If we choose $C_\Ge=3K C_{\Ge,M}$, then
$$
D^d_p(\Gd)\leq C_\Ge \Gd^{-2\Ge},
$$
which finishes the proof.
\end{proof}

We can finally give a proof of Theorem \ref{uniform}, using a simple rescaling argument that will also be used extensively onwards. Given an arbitrary $\phi\in \mathcal D_d$. Put $\psi:=\norm{\phi''}_\infty^{-1}\phi\in \mathcal A_d$. If $\mathcal P$ is sub-admissible for $\phi$ at scale $\Gd$, then it is sub-admissible for $\psi$ at scale $\norm{\phi''}_\infty^{-1}\Gd$. 

Given $f\in L^p(\R^2)$ with Fourier support in $\mathcal N^\phi_{[0,1],\Gd}$. Define $\hat g(s,t)=\hat f(s,\norm{\phi''}_\infty^{-1}t)$. Then $g\in L^p(\R^2)$ with Fourier support in $\mathcal N^\psi_{[0,1],\norm{\phi''}_\infty^{-1}\Gd}$. Applying Theorem \ref{familyA} to $g$ at scale $\norm{\phi''}_\infty^{-1}\Gd$ to get
$$
\norm {g}_{L^p(\R^2)}\leq  C_\Ge \norm{\phi''}_\infty^{\Ge}\Gd^{-\Ge}\left(\sum_{I\in\mathcal P}\norm {g_I }^2_{L^p(\R^2)}\right)^{\frac 1 2}.
$$
Rescaling back, we obtain the desired conclusion for $f$.

\subsection{Proof we Lemma \ref{lemboot}}

\begin{proof}
Since a rescaled polynomial of degree at most $d$ is also a polynomial of degree at most $d$, it suffices to consider $J=[0,1]$. Also, without loss of generality we may assume $\sup_{s\in [0,1]}|\phi''(s)|=1$. Then the sum $\Gl$ of the absolute values of all coefficients of $\phi''$ is bounded by a constant depending on $d$ only (this is a consequence of an inequality of A. and M. Markov (see, for instance, \cite{Shadrin}).

By the fundamental theorem of algebra, the set $B(\phi,[0,1])$ is a union of at most $O(d)$ intervals that are relatively open in $[0,1]$. By Proposition 2.2 of \cite{MR1879821} in the case $n=1$ applied to $\phi''$, each such interval has length 
$$
\lesssim_d \Gs^{\frac 1{d-2}} \Gl^{-\frac 1{d-2}}\sim_d \Gs^{\frac 1{d-2}},
$$
from which the result follows if $C_d$ is large enough.
\end{proof}

\section{More on admissible partitions}\label{sectionadmis}
In this section we discuss some further properties of admissible partitions that will be used later.

We start with a simple observation. Let $I_0$ be a compact interval. If $\phi$ is linear on $I_0$, then any partition is super-admissible for $\phi$ at any scale $r>0$, while the trivial partition is the only sub-admissible and thus the only admissible partition for $\phi$ at scale $r$.

The sharpest case of Theorem \ref{uniform} is when $\mathcal P$ is actually admissible (i.e. both super-admissible and sub-admissible). The following proposition says that every $C^2$ phase function on $[0,1]$ admits an admissible partition at any scale $\Gd>0$, so each phase function has a sharp decoupling inequality in terms of the partition.
\begin{prop}\label{exist}
Let $I_0$ be compact and $\phi:I_0\to \R$ be $C^2$. Then for any $r>0$, there exists an admissible partition $\mathcal P$ of $I_0$ for $\phi$ at scale $r>0$.
\end{prop}
\begin{proof}
If $\phi$ is linear, then the trivial partition is admissible. Otherwise, $\norm {\phi''}_\infty>0$ and we will construct an admissible partition $\mathcal P$ of $I_0$ for $\phi$ at scale $r$.

Denote $I_0=[\Ga,\Gb]$ and let $a_0=\Ga$. Let 
$$
a_1:=\max\{t\in [a_0,\Gb]:\max_{s,c\in [a_0,t]}|\phi(s)-\phi(c)+\phi'(c)(s-c)|\leq 2r\}.
$$
Such $a_1$ always exists in $[a_0,\Gb]$ since $\phi$ is $C^2$. If $a_1=\Gb$, then we arrive at the trivial partition which is admissible for $\phi$ at scale $r$.

If $a_1<\Gb$, then there are $s,c\in [a_0,a_1]$, either $s$ or $c$ being $a_1$, such that
$$
|\phi(s)-\phi(c)+\phi'(c)(s-c)|= 2r.
$$
Taylor's theorem implies that 
$$
\norm {\phi''}_\infty (s-c)^2\geq 4r,
$$
so 
$$
a_1-a_0\geq |s-c|\geq \frac {2r^{1/2}}{\norm {\phi''}^{1/2}_\infty}.
$$
Let
$$
a_2:=\max\{t\in [a_1,\Gb]:\max_{s,c\in [a_1,t]}|\phi(s)-\phi(c)+\phi'(c)(s-c)|\leq 2r\}.
$$
If $a_2=\Gb$, then $\mathcal P:=\{[a_0,a_1],[a_1,a_2]\}$ is an admissible partition of $[\Ga,\Gb]$ for $\phi$ at scale $r$. Otherwise, by the same analysis above, we have $a_2-a_1\geq 2\sqrt{r/\norm {\phi''}_\infty}$.
Then define $a_3$ in the above fashion, and repeat.

This process must stop at finite time since $a_{n}-a_{n-1}\geq 2\sqrt{r/\norm {\phi''}_\infty}$ for all $n\geq 1$. 
\end{proof}

The next proposition is about the lengths of intervals constituting a sub-admissible partition for certain phase functions.
\begin{prop}\label{number}
Let $\phi:I_0\to \R$ be $C^2$. Let $\Gd>0$ and suppose $\mathcal P$ is a sub-admissible partition of $I_0$ for $\phi$ at scale $\Gd$. Then there is a partition $\mathcal P'$ of $I_0$, such that each interval $I\in \mathcal P'$, except possibly the last one, is a union of two adjacent intervals in $\mathcal P$ and has length bounded below by
    $$
    2\sqrt{\frac {\Gd}{\norm {\phi''}_\infty}}.
    $$
As a result, the number of intervals in $\mathcal P$ is bounded above by $\Gd^{-1/2}{\norm {\phi''}^{1/2}_\infty}+1$.
\end{prop}
\begin{proof}
Denote $\mathcal P$ as $a_j$, $0\leq j\leq n$. If $\mathcal P$ is the trivial partition, then we define $\mathcal P'$ to be trivial as well. If not, then in particular, $\phi$ cannot be linear, so $\norm {\phi''}_\infty>0$. By \eqref{Taylor}, we have
$$
a_{j+2}-a_j\geq \frac {2\Gd^{1/2}}{\norm {\phi''}^{1/2}_\infty}. 
$$
Therefore, we can define $\mathcal P'$ as follows. If $n$ is even, then we define $\mathcal P'=\{[a_{2k},a_{2(k+1)}]:0\leq k\leq n/2 \}$. If $n$ is odd, then we define $\mathcal P'=\{[a_{2k},a_{2(k+1)}]:0\leq k\leq (n-3)/2 \}\cup \{[a_{n-1},a_n]\}$. Thus $\mathcal P'$ is as required. 

The bound on the number of intervals in $\mathcal P$ follows immediately.
\end{proof}

\section{Decoupling for curves with nonvanishing curvature}\label{secnonzerocurve}
The results in this section will serve as the base case of our induction on scales. Our starting point will be Bourgain-Demeter's classical decoupling inequality \cite{MR3374964}.
\begin{thm}\label{BD15}
Let $2\leq p\leq 6$, $\Gd\in 2^{-2\N}$, and $M\geq 1$. Then for any $f\in L^p(\R^2)$ with Fourier support in $\mathcal N^{s^2}_{[0,1],M\Gd}$, we have
$$
\norm {f}_{L^p(\R^2)}\lesssim_{\Ge,M} \Gd^{-\Ge}\left(\sum_{j=1}^{\Gd^{-1/2}} \norm{f_{[(j-1)\Gd^{1/2},j\Gd^{1/2}]}}^2_{L^p(\R^2)}\right)^{\frac 1 2}.
$$
\end{thm}
{\it Remark.} The introduction of the constant $M$ here is just for technical reasons. In \cite{MR3374964} the authors only mentioned the case $M=1$, but one can slightly adapt the proof of Theorem 5.1 of \cite{MR3592159} to prove Theorem \ref{BD15}. We also refer the interested reader to \cite{MR3971577} and \cite{MR4024357} for the technical details.

\subsection{Uniform decoupling for quadratic polynomials, canonical partition}
We first slightly generalise this theorem to all quadratic functions $\phi$, with the same dyadic partition. This is achieved by the following simple lemma.
\begin{lem}\label{lemlin}
Let $1<p<\infty$, $r>0$ and $\mathcal P$ be any finite partition of a compact interval $I_0$. For any $C^2$ function $\phi:I_0\to \R$, let $K(\phi,r)$ be the best constant such that for all $f\in L^p(\R^2)$ with Fourier support in $\mathcal N^{\phi}_{I_0,r}$, we have
$$
\norm {f}_{L^p(\R^2)}\leq K(\phi,r)\left(\sum_{I\in\mathcal P} \norm{f_{I}}^2_{L^p(\R^2)}\right)^{\frac 1 2}.
$$
Then for any linear function $l(s)=cs+d$ and any $\Gl>0$, we have
$$
K(\phi,r)=K(\Gl\phi+l,\Gl^{-1}r).
$$
\end{lem}
\begin{proof}
We first prove the case $\Gl=1$. It suffices to show $K(\phi+l,r)\leq K(\phi,r)$. Let $f\in L^p(\R^2)$ with Fourier support in $\mathcal N^{\phi+l}_{I_0,r}$. Then we define $g$ by the relation
$$
\hat g(s,t)=\hat f(s,t+l(s)).
$$
Then $g\in L^p(\R^2)$ with Fourier support in $\mathcal N^{\phi}_{I_0,r}$. Applying the definition of $K(\phi)$, we have
$$
\norm {g}_{L^p(\R^2)}\leq K(\phi,r)\left(\sum_{I\in\mathcal P} \norm{g_{I}}^2_{L^p(\R^2)}\right)^{\frac 1 2}.
$$
By direct computation, for any interval $I\sub [0,1]$, we have
$$
g_I(x,y)=e(-yd)f_I(x-cy,y),
$$
which shows that $\norm {g_I}_{L^p(\R^2)}=\norm {f_I}_{L^p(\R^2)}$ since the Jacobian determinant is $1$. Hence $K(\phi+l,r)\leq K(\phi,r)$ and our result follows.

For other $\Gl>0$, we apply a rescaling. Let $f\in L^p(\R^2)$ with Fourier support in $\mathcal N^{\Gl\phi+l}_{I_0,r}$. Define
$$
\hat g (s,t)=\hat f(s,\Gl t).
$$
Then $g\in L^p(\R^2)$ with Fourier support in $\mathcal N^{\phi+\Gl^{-1}l}_{I_0,\Gl^{-1}r}$. Apply the definition of $K(\phi,r)$ to $g$ and the case $\Gl=1$. Rescaling back, the result then follows.
\end{proof}
Combining Theorem \ref{BD15} and Lemma \ref{lemlin} gives the following corollary.
\begin{cor}[Uniform decoupling for quadratic polynomials, canonical partition]\label{bs2}
Let $2\leq p\leq 6$ and $M\geq 1$. Let $\phi(s)=bs^2+cs+d$ where $b\leq M$ and $c,d\in \R$. Then for any $\Gd\in 2^{-2\N}$ and any $f\in L^p(\R^2)$ with Fourier support in $\mathcal N^{\phi}_{[0,1],Mb\Gd}$, we have
$$
\norm {f}_{L^p(\R^2)}\lesssim_{\Ge,M} \Gd^{-\Ge}\left(\sum_{j=1}^{\Gd^{-1/2}} \norm{f_{[(j-1)\Gd^{1/2},j\Gd^{1/2}]}}^2_{L^p(\R^2)}\right)^{\frac 1 2}.
$$
\end{cor}
\subsection{Decoupling for curves with nonvanishing curvature, canonical partition}
In \cite{MR3374964}, the authors generalised the decoupling inequality to $C^2$-hypersurfaces with positive principle curvatures, using a Pramanik-Seeger type iteration, which is also used in \cite{MR4078083}. However, for our purposes, it is better for us to keep track of how the decoupling constants actually depend on the geometric quantities of the curves. So we will be essentially reproducing their proof, but with more careful consideration of the constant factors.

\begin{lem}\label{nonzero}
Let $2\leq p\leq 6$, $\Gd\in 2^{-2\N}$ and $M>1$. Let $\phi:[0,1]\to \R$ be a $C^3$ function with $\inf_{s\in [0,1]}|\phi''(s)|\geq M^{-1}$, $\norm {\phi''}_\infty \leq M$ and $\norm {\phi'''}_\infty \leq M$. Then for any $f\in L^p(\R^2)$ with Fourier support in $\mathcal N^{\phi}_{[0,1],M\Gd}$, we have (where we denote $a_j=j\Gd^{1/2}$, $0\leq j\leq \Gd^{-1/2}$)
$$
\norm {f}_{L^p(\R^2)}\lesssim_{\Ge,M} \Gd^{-\Ge}\left(\sum_{j=1}^{\Gd^{-1/2}} \norm{f_{[a_{j-1},a_j]}}^2_{L^p(\R^2)}\right)^{\frac 1 2}.
$$
\end{lem}

\begin{proof}
Let $K(\Gd)=K(\Gd,M)$ be the best constant such that for any $f\in L^p(\R^2)$ with Fourier support in $\mathcal N^{\phi}_{[0,1],M\Gd}$, we have
$$
\norm {f}_{L^p(\R^2)}\leq K(\Gd)\left(\sum_{j=1}^{\Gd^{-1/2}} \norm{f_{[a_{j-1},a_j]}}^2_{L^p(\R^2)}\right)^{\frac 1 2}.
$$
Our goal is to show that $K(\Gd)\lesssim_{\Ge,M}\Gd^{-\Ge}$.

Let $\Gd'\in 2^{-2\N}$, $\Gd\leq \Gd'\leq 1$ be an intermediate scale to be determined. Since $f$ is Fourier supported in $\mathcal N^{\phi}_{[0,1],M\Gd}$ and $\Gd'\geq \Gd$, it is Fourier supported in $\mathcal N^{\phi}_{[0,1],M\Gd'}$ as well. Apply the definition of $K(\Gd')$ to get (where we denote $b_k=k\Gd'^{1/2}$, $0\leq k\leq \Gd'^{-1/2}$)
$$
\norm {f}_{L^p(\R^2)}\leq K(\Gd')\left(\sum_{k=1}^{\Gd'^{-1/2}} \norm{f_{[b_{k-1},b_k]}}^2_{L^p(\R^2)}\right)^{\frac 1 2}.
$$
Our hope is that for each $1\leq k\leq \Gd'^{-1/2}$, the graph of $\phi$ over $[b_{k-1},b_k]$ is approximated by a parabola with error $O(\Gd)$. We define 
$$
p_k(s)=\phi(b_{k-1})+\phi'(b_{k-1})(s-b_{k-1})+\frac 1 2\phi''(b_{k-1})(s-b_{k-1})^2.
$$
Then by Taylor expansion, for $s\in [b_{k-1},b_k]$, we have
$$
|\phi(s)-p_k(s)|\leq \frac{\norm {\phi'''}_\infty}{6} \Gd'^{3/2}\leq  M \Gd'^{3/2}.
$$
This suggests that we take $\Gd'$ to be the smallest number in $2^{-2\N}$ such that $\Gd'\geq\Gd^{2/3}$.

With this choice of $\Gd'$ (so $\Gd'<4\Gd^{2/3}$), since $f$ is Fourier supported in $\mathcal N^{\phi}_{[0,1],M\Gd}$, we have 
$$
\mathrm{supp}(\hat f_{[b_{k-1},b_k]})\sub \mathcal N^{p_k}_{[b_{k-1},b_k],9M\Gd}\sub \mathcal N^{p_k}_{[b_{k-1},b_k],18M^2 \phi''(b_{k-1}) \Gd/2},
$$
where we have used the assumption that $\inf_{s\in [0,1]}\phi''(s)\geq M^{-1}$.

Applying Corollary \ref{bs2} with $M^2$ and $b=\phi''(b_{k-1})/2\leq M^2$, we get
$$
\norm {f_{[a_{k-1},a_k]}}_{L^p(\R^2)}\lesssim_{\Ge,M}\Gd^{-\Ge}\left(\sum_{j:[a_{j-1},a_j]\sub [b_{k-1},b_k]} \norm{f_{[a_{j-1},a_j]}}^2_{L^p(\R^2)}\right)^{\frac 1 2}.
$$
Squaring both sides and summing over $k$, we obtain
$$
\norm f_{L^p(\R^2)}\lesssim_{\Ge,M}K(\Gd')\Gd^{-\Ge}\left(\sum_{j=1}^{\Gd^{-1/2}} \norm{f_{[a_{j-1},a_j]}}^2_{L^p(\R^2)}\right)^{\frac 1 2}.
$$
This implies that 
$$
K(\Gd)\leq C_{\Ge,M}\Gd^{-\Ge}K(\Gd'),
$$
where we recall $\Gd'$ is the smallest number in $2^{-2\N}$ such that $\Gd'\geq \Gd^{2/3}$. For $\Gd<1/4$, we iterate this inequality $n$ times until we get to the scale $1/4$:
$$
K(\Gd)\leq C_{\Ge,M}^n \Gd^{-\Ge \left(1+\frac 2 3+\dots+\left(\frac 2 3\right)^{n-1}\right)}K\left(\frac 1 4\right)\leq CC_{\Ge,M}^n \Gd^{-3\Ge},
$$
since $K(1/4)\sim 1$ by triangle inequality and Cauchy-Schwarz. Lastly, since $\Gd'\geq\Gd^{2/3}$ in each iteration and $n$ is the first time we stop the iteration, we have $\Gd^{(2/3)^{n-1}}< 1/4$.
This shows that 
$$
n<1+\frac {\log\log (\Gd^{-1})-\log \log 4}{\log (3/2)}\leq C\log\log (\Gd^{-1}),
$$
for some suitable absolute constant $C$. Thus
$$
C_{\Ge,M}^n\leq (\log(\Gd^{-1}))^{C\log C_{\Ge,M}}\lesssim_{\Ge,M} \Gd^{-\Ge},
$$
and hence we have $K(\Gd)\lesssim_{\Ge,M}\Gd^{-4\Ge}$.
\end{proof}

\subsection{Decoupling for curves with nonvanishing curvature}
We may finally upgrade Lemma \ref{nonzero} to the case of an arbitrary sub-admissible partition. For future use, we first introduce another notation.
\begin{defn}[Decoupling for a phase function]\label{decouplingsinglefunction}
Let $1<p<\infty$. Let $\phi:[0,1]\to \R$ be $C^2$. For $\Gd>0$, let $D^\phi_p(\Gd)$ be the best constant such that for each $f\in L^p(\R^2)$ with Fourier support in $\mathcal N^\phi_{[0,1],\Gd}$, we have
$$
\norm {f}_{L^p(\R^2)}\leq D^\phi_p(\Gd)\left(\sum_{I\in\mathcal P}\norm {f_I }^2_{L^p(\R^2)}\right)^{\frac 1 2},
$$
where $\mathcal P$ is any sub-admissible partition of $[0,1]$ for $\phi$ at scale $\Gd$.
\end{defn}

\begin{thm}\label{quaddom}
Let $2\leq p\leq 6$. Let $M\geq 1$ and let $\phi:[0,1]\to \R$ be a $C^3$ function with
\begin{equation}\label{curvaturegood}
     \norm{\phi'''}_{\infty}+\norm{\phi''}_\infty\leq M\inf_{s\in [0,1]}|\phi''(s)|\quad\text{and}\quad \norm {\phi''}_\infty\leq M.
\end{equation}
Then for any $\Ge>0$, there is some $C_{\Ge,M}$ such that $D^{\phi}_p(\Gd)\leq  C_{\Ge,M}\Gd^{-\Ge}$ for any $0<\Gd\leq 1$.
\end{thm}
The rest of this subsection is devoted to the proof of this theorem. The main ingredients of the proof include Proposition \ref{number}, Lemma \ref{nonzero} and the following simple tiling argument.
\begin{prop}\label{tiling}
Let $0<l_0\leq 1/4$ and let $\mathcal P$ be a collection of disjoint subintervals of $[0,1]$ with lengths bounded above by $2l_0$ and below by $l_0$. Then there is $l\in 2^{-\N}$ with $l/l_0\in [4,8)$ and two subcollections $\mathcal U_i$, $i=1,2$ of $\mathcal P$, such that the following statements are true.
\begin{enumerate}
    \item For each $I\in \mathcal U_1$, there is some $1\leq j\leq l^{-1}$ such that $I\sub [(j-1)l,jl]$. Moreover, each such $[(j-1)l,jl]$ contains less than $8$ intervals $I$.
    \item For each $I\in \mathcal U_2$, there is some $1\leq j\leq l^{-1}$ such that $I\sub [(j-1/2)l,(j+1/2)l]\cap [0,1]$. Moreover, each such $[(j-1/2)l,(j+1/2)l]$ contains less than $8$ intervals $I$.
\end{enumerate}

\end{prop}

\begin{proof}
Let $l\in 2^{-\N}$ be the smallest number such that $l\geq 4l_0 $, so $l/l_0\in [4,8)$. Each interval $I\in \mathcal P$ has length at most $2l_0\leq l/2$. Include $I$ inside $\mathcal U_1$ if it is fully contained in a dyadic interval $[(j-1)l,jl]$ for some $1\leq j\leq l^{-1}$. Otherwise, it has to be fully contained in $[(j-1/2)l,(j+1/2)l]$ for some $1\leq j\leq l^{-1}$, so we can include it in the collection $\mathcal U_2$. The bound on the number of intervals $I\in \mathcal P$ contained in each dyadic interval follows from the lower bound of the lengths of the intervals $I$.
\end{proof}

Now we can give a proof of Theorem \ref{quaddom}, in a series of steps.
\subsubsection{A few technical reductions}
If $\inf_{s\in [0,1]}|\phi''(s)|=0$, then by \eqref{curvaturegood}, we have $\norm {\phi''}_\infty=0$, so $\phi$ is linear and thus $D^\phi_p(\Gd)=1$. If $\inf_{s\in [0,1]}|\phi''(s)|>0$, since $\norm{\phi''}_\infty\leq M$, by rescaling in the vertical axis as in the proof of Lemma \ref{lemlin}, we may assume $\inf_{s\in [0,1]}\phi''(s)=1$.

Let $f\in L^p(\R^2)$ with Fourier support in $\mathcal N^{\phi}_{[0,1],\Gd}$ and $\mathcal P$ be a sub-admissible partition of $[0,1]$ for $\phi$ at scale $\Gd$. We may also assume $\Gd<M$, otherwise $\mathcal P$ is trivial and the bound $1$ works.

We invoke Proposition \ref{number} to get the coarser partition $\mathcal P'$. Since $\norm {\phi''}_\infty\leq M$, each interval $I\in \mathcal P'$, except possibly the last one, is a union of two adjacent intervals in $\mathcal P$ and has length bounded below by $2(\Gd/M)^{1/2}$. As a result, the number of intervals in $\mathcal P$ is bounded above by $(M/\Gd)^{1/2}+1$.

If the number of intervals in $\mathcal P$ is odd, then by triangle inequality and Cauchy-Schwarz we can disregard the Fourier truncation of $f$ over the last interval of $\mathcal P$. Thus, for the rest of the proof we may assume that the number of intervals in $\mathcal P$ is even and that each interval in $\mathcal P'$ has length bounded below by $2(\Gd/M)^{1/2}$. Since each interval $I\in \mathcal P'$ is a union of two adjacent intervals in $\mathcal P$, by triangle inequality and Cauchy-Schwarz it suffices to prove that
$$
\norm {f}_{L^p(\R^2)}\lesssim_{\Ge,M} \Gd^{-\Ge}\left(\sum_{I\in \mathcal P'} \norm{f_{I}}^2_{L^p(\R^2)}\right)^{\frac 1 2}.
$$

We now partition the collection $\mathcal P'$ according to the lengths of intervals. Let $I^*$ be an interval in $\mathcal P'$ with maximum length (we may of course assume $|I^*|\leq 1/2$). Let $\mathcal P'_1$ be the collection of intervals in $\mathcal P'$ with length $>|I^*|/2$. For each $k\geq 2$, let $\mathcal P'_k$ be the collection of intervals in $\mathcal P'$ with length in the range $(2^{-k}|I^*|,2^{-k+1}|I^*|]$. Since each interval in $\mathcal P'$ has length bounded below by $2(\Gd/M)^{1/2}$, we have only $O_M(\log (\Gd^{-1}))$ many such collections. Since we can afford logarithmic losses, it suffices to show for each $\mathcal P'_k$ that
$$
\norm {f_{\cup \{I:I\in \mathcal P'_k\}}}_{L^p(\R^2)}\lesssim_{\Ge,M} \Gd^{-\Ge}\left(\sum_{I\in \mathcal P'_k} \norm{f_{I}}^2_{L^p(\R^2)}\right)^{\frac 1 2}.
$$
Now fix such $k\geq 1$. We can apply Proposition \ref{tiling} with $l_0=2^{-k}|I^*|\leq 1/4$ and $\mathcal P=\mathcal P'_k$ to get the corresponding $l=l(k)$ and $\mathcal U_i=\mathcal U_i(k)$, $i=1,2$. Also, note that $l\geq 8(\Gd/M)^{1/2}$ since $l_0\geq 2(\Gd/M)^{1/2}$.

By triangle inequality and Cauchy-Schwarz again, it suffices to prove for $i=1,2$ that
$$
\norm {f_{J_i}}_{L^p(\R^2)}\lesssim_{\Ge,M} \Gd^{-\Ge}\left(\sum_{I\in \mathcal U_i} \norm{f_{I}}^2_{L^p(\R^2)}\right)^{\frac 1 2},
$$
where $J_i:=\cup \{I:I\in \mathcal U_i\}$, $i=1,2$. 

\subsubsection{Applying Lemma \ref{nonzero}}
We deal with $i=1$ first. We have $\Gd\leq Ml^2$ by our choice of $l$. By \eqref{curvaturegood}, we may apply Lemma \ref{nonzero} with scale $\Gd=l^2$ to get
$$
\norm {f_{J_1}}_{L^p(\R^2)}\lesssim_{\Ge,M} l^{-\Ge}\left(\sum_{j=1}^{l^{-1}} \norm{f_{[(j-1)l,jl]\cap J_1}}^2_{L^p(\R^2)}\right)^{\frac 1 2}.
$$
But by Proposition \ref{tiling}, for each $j$, $[(j-1)l,jl]\cap J_1$ is equal to a union of less than $8$ intervals $I\in \mathcal U_1$. By triangle inequality and Cauchy-Schwarz, we have
$$
\norm {f_{J_1}}_{L^p(\R^2)}\lesssim_{\Ge,M} l^{-\Ge}\left(\sum_{I\in \mathcal U_1} \norm{f_{I}}^2_{L^p(\R^2)}\right)^{\frac 1 2}\lesssim_{\Ge,M}\Gd^{-\Ge}\left(\sum_{I\in \mathcal U_1} \norm{f_{I}}^2_{L^p(\R^2)}\right)^{\frac 1 2}.
$$
For the case $i=2$, we note that translating the domain of $\phi$ to the right by $\Gs:=l/2$ is equivalent to changing $\phi(s)$ to $\tilde \phi(s):=\phi(s+\Gs)$, $s\in [0,1-\Gs]$. Since the domain of $\tilde \phi$ is now a subset of $[0,1]$, the conditions in \eqref{curvaturegood} still hold. Hence, the same argument for the case $i=1$ works in this case.

The conclusion for quadratic polynomials follows immediately if we take $M$ to be the absolute constant that bounds the quadratic coefficient.

\section{A rescaling theorem}\label{secpolyres}

We need a little more notation. If $\mathcal P$ is a partition of a compact interval $I_0$ and $J$ is a union of consecutive intervals in $\mathcal P$, we denote by $\mathcal P(J)$ the partition of $J$ using the intervals in $\mathcal P$. We then prove the following rescaling theorem.

\begin{thm}[Rescaling]\label{cubicres}

Let $1<p<\infty$. Let $\Gd>0$, $\phi$ be $C^2$ and $\mathcal P$ be 
a sub-admissible partition of $[0,1]$ for $\phi$ at scale $\Gd$. Then for any $J=[\Ga,\Gb]$ which is a union of consecutive intervals in $\mathcal P$, there exists another $C^2$-function $\psi$, such that for any $f\in L^p(\R^2)$ with Fourier support in $\mathcal N^\phi_{J,\Gd}$, 
\begin{equation}\label{cubicresineq}
    \norm {f}_{L^p(\R^2)}\leq D^\psi_p((\Gb-\Ga)^{-1}\Gd)\left(\sum_{I\in\mathcal P(J)}\norm {f_I }^2_{L^p(\R^2)}\right)^{\frac 1 2},
\end{equation}
where $D_p^\psi(\Gd)$ was defined in \eqref{decouplingsinglefunction}.

In particular, if $\phi\in \mathcal A_d$ (as defined at the beginning of Section \ref{secproofmain}), then $\psi$ can be also taken to be in $\mathcal A_d$. In this case, we also have
\begin{equation}\label{gencubicresineq}
    \norm {f}_{L^p(\R^2)}\leq D^d_p((\Gb-\Ga)^{-1}\Gd)\left(\sum_{I\in\mathcal P(J)}\norm {f_I }^2_{L^p(\R^2)}\right)^{\frac 1 2},
\end{equation}
where $D^d_p(\Gd)$ was defined in \eqref{defdecAd}.

\end{thm}

\begin{proof}
By a change of variables, we have
$$
f_{[\Ga,\Gb]}(x,y)=\int_{-\Gd}^{\Gd}\int_\Ga^\Gb \hat f(s,\phi(s)+t)e(x s+y(\phi(s)+t))ds dt.
$$
Define $s'=(s-\Ga)/(\Gb-\Ga)\in [0,1]$. Then by direct computation,
\begin{align*}
    f_{[\Ga,\Gb]}(x,y)
    &=(\beta-\Ga)\int_{-\Gd}^{\Gd}e(ty)\int_0^1 \hat f(s,\phi(s)+t)\\
    &\cdot e(x(\Gb-\Ga)s'+\Ga)\\
    &\cdot e\left(y \phi(\Ga+(\Gb-\Ga)s')\right)ds'dt.
\end{align*}

We define $\psi$ by
\begin{equation}\label{defpsi}
    \psi(s')=(\Gb-\Ga)^{-1}\phi(\Ga+(\Gb-\Ga)s').
\end{equation}
Thus $\psi(s')=(\Gb-\Ga)^{-1}\phi(s)$.

Define $t'=(\Gb-\Ga)^{-1}t$ and $(x',y')=(\Gb-\Ga)(x,y)$. We also define another function $F$ by the relation $\hat F(s',\psi(s')+t')=\hat f(s,\phi(s)+t)$. More explicitly, for any $(u,v)\in \R^2$, $\hat F$ is defined as
$$
\hat F(u,v)=\hat f((\Gb-\Ga)u+\Ga,(\Gb-\Ga)v).
$$
Then we see that $F\in L^p(\R^2)$ and is Fourier supported on $\mathcal N^\psi_{[0,1],(\Gb-\Ga)^{-1}\Gd}$. Thus, in the above notation, we arrive at 
\begin{align*}
  f_{[\Ga,\Gb]}(x,y)
&=e(\Ga x)\int_{-\frac{\Gd}{(\Gb-\Ga)}}^{\frac{\Gd}{(\Gb-\Ga)}}e(t'y')\int_0^1 \hat F(s',\psi(s')+t') e(x's')e(y'\psi(s'))ds' dt'\\
&=e(\Ga x)F(x',y').
\end{align*}
Also, observe that the following partition of $[0,1]$
$$
\mathcal P':=\left\{I'=\frac {I-\Ga}{\Gb-\Ga}:I\in \mathcal P\right\}
$$
is sub-admissible for $\psi$ at scale $(\Gb-\Ga)^{-1}\Gd$. Applying the definition of $D^\psi_p((\Gb-\Ga)^{-1}\Gd)$, we have
$$
\norm {F}_{L^p(\R^2)}\leq D^\psi_p((\Gb-\Ga)^{-1}\Gd)\left(\sum_{I'\in\mathcal P'}\norm {F_{I'} }^2_{L^p(\R^2)}\right)^{\frac 1 2}.
$$
Rescaling back (as in the proof of Lemma \ref{lemlin}), we obtain
$$
\norm {f}_{L^p(\R^2)}\leq D^\psi_p((\Gb-\Ga)^{-1}\Gd)\left(\sum_{I\in\mathcal P(J)}\norm {f_{I} }^2_{L^p(\R^2)}\right)^{\frac 1 2},
$$
which proves \eqref{cubicresineq}.

To prove \eqref{gencubicresineq}, it suffices to show that $\psi\in \mathcal A_d$ whenever $\phi\in \mathcal A_d$. For the bound on the second derivative, we have
$$
\psi''(s')=(\Gb-\Ga)\phi''(\Ga+(\Gb-\Ga)s'),
$$
which, together with the fact that $0\leq \Ga<\Gb\leq 1$, implies $\norm {\psi''}_\infty\leq \norm {\phi''}_\infty\leq 1$. 
It remains to show that for any $0<\Gs<C_d^{-d}$ and any interval $K\sub [0,1]$, the bad set $B(\psi,K)$ (as defined in \eqref{defB}) is a union of at most $C_d$ disjoint subintervals relatively open in $K$ and also satisfies $|B(\psi,K)|\leq C_d \Gs^{1/d}|K|$. But by definition, we have exactly
$$
B(\psi,K)=\left\{\frac{s-\Ga}{\Gb-\Ga}:s\in B(\phi,(\Gb-\Ga)K+\Ga)\right\}.
$$
Since $\phi\in \mathcal A_d\sub \mathcal D_d$, we have $B(\phi,(\Gb-\Ga)K+\Ga)$ is a union of at most $C_d$ disjoint subintervals relatively open in $(\Gb-\Ga)K+\Ga$ and also satisfies 
$$
|B(\phi,(\Gb-\Ga)K+\Ga)|\leq C_d \Gs^{1/d}(\Gb-\Ga)|K|.
$$
Thus $B(\psi,K)$ is also a union of at most $C_d$ disjoint subintervals relatively open in $K$ and also satisfies $|B(\psi,K)|\leq C_d \Gs^{1/d}|K|$.

\end{proof}

\section{Proof of the bootstrap inequality}\label{bootproof}

Using Theorem \ref{quaddom} and Theorem \ref{cubicres}, we now prove Theorem \ref{mainbootthm}. The main idea is to partition $[0,1]$ into subintervals according as whether $|\phi''|$ is large. On subintervals where $|\phi''|$ is large, we use decoupling for curves with nonvanishing curvature, which is Theorem \ref{quaddom}. On subintervals where $|\phi''|$ is small, we use the rescaling Theorem \ref{cubicres}, and use the assumption in Definition \ref{polylike} that such intervals have small total length.

\begin{proof}
Let $\phi\in \mathcal A_d=\mathcal A_d(C_d)$ be given, and let $M>C_d$. We will find $K=K(d)$ and $C_{\Ge,M}$ such that
\begin{equation}\label{hypoeqn}
    D^\phi_p(\Gd)\leq K(C_{\Ge,M} \Gd^{-\Ge}+ \sup_{\Gd'\geq M\Gd}D^d_p(\Gd')).
\end{equation}
Let $\Gd>0$ and $\mathcal P$ be a sub-admissible partition of $[0,1]$ for $\phi$ at scale $\Gd$. Let $f\in L^p(\R^2)$ with Fourier support in $\mathcal N^{\phi}_{[0,1],\Gd}$. 

Take $B=B(\phi,[0,1])$ as in \eqref{defB} with $\Gs:=C_d^{-d}M^{-d}$. Split $\mathcal P$ into 3 subcollections
\begin{align*}
    &\mathcal P_1:=\{I\in \mathcal P:I\sub B\}\\
    &\mathcal P_2:=\{I\in \mathcal P:I\sub [0,1]\bs B\}\\
    &\mathcal P_3:=\{I\in \mathcal P:I\cap B\neq \varnothing\text{ and } I\bs B\neq \varnothing\}.
\end{align*}
Denote $f_i:=f_{\cup \{I:I\in \mathcal P_i\}}$, $i=1,2,3$.

By definition of $\mathcal D_d$, $\mathcal P_3$ has cardinality bounded above by $C_d$, so by triangle inequality and Cauchy-Schwarz it suffices to consider $\mathcal P_1$ and $\mathcal P_2$.

Consider $\mathcal P_1$ first. Write $B=\cup_i^N J_i$ where $N\leq C_d$ and $|J_i|\leq |B|\leq C_d\Gs^{1/d}=M^{-1}$. Apply \eqref{gencubicresineq} to $J_i$ to get
$$
\norm {(f_1)_{J_i}}_{L^p(\R^2)}\leq D^d_p(|J_i|^{-1}\Gd)\left(\sum_{I\in\mathcal P(J_i)}\norm {f_I }^2_{L^p(\R^2)}\right)^{\frac 1 2}\leq \sup_{\Gd'\geq M\Gd}D^d_p(\Gd')\left(\sum_{I\in\mathcal P(J_i)}\norm {f_I }^2_{L^p(\R^2)}\right)^{\frac 1 2}.
$$
Thus
$$
\norm {f_1}_{L^p(\R^2)}\lesssim_d \sup_{\Gd'\geq M\Gd}D^d_p(\Gd')\left(\sum_{I\in\mathcal P_1}\norm {f_I }^2_{L^p(\R^2)}\right)^{\frac 1 2}.
$$
Now we come to $\mathcal P_2$. Write $[0,1]\bs B=\cup_i^{N'}J'_i$ where $N'\leq C_d+1$.

Apply \eqref{cubicresineq} to $J'_i:=[\Ga,\Gb]$ to get
$$
\norm {(f_2)_{J'_i}}_{L^p(\R^2)}\leq D^\psi_p(|J_i|^{-1}\Gd)\left(\sum_{I\in\mathcal P(J'_i)}\norm {f_I }^2_{L^p(\R^2)}\right)^{\frac 1 2},
$$
where $\psi(s)=(\Gb-\Ga)^{-1}\phi(\Ga+(\Gb-\Ga)s)$ as in \eqref{defpsi}.

By similar argument at the end of Section \ref{secpolyres}, since $\phi\in \mathcal A_d$, we have $\psi$ satisfies both inequalities of \eqref{curvaturegood} with $M$ replaced by $C_dM^d$. Hence, by Theorem \ref{quaddom}, we have for $0<\Gd\leq 1$ that
$$
D^\psi_p(|J_i|^{-1}\Gd)\lesssim_{\Ge,M,p}(|J_i|^{-1}\Gd)^{-\Ge}\leq \Gd^{-\Ge}.
$$
Thus, by triangle inequality and Cauchy-Schwarz,
$$
\norm {f_2}_{L^p(\R^2)}\lesssim_{\Ge,d,M,p} \Gd^{-\Ge}\left(\sum_{I\in\mathcal P_2}\norm {f_I }^2_{L^p(\R^2)}\right)^{\frac 1 2}.
$$
Combining the estimates above, we have
$$
\norm {f}_{L^p(\R^2)}\leq K\left( C_{\Ge,d,M,p}\Gd^{-\Ge}+\sup_{\Gd'\geq M\Gd}D^d_p(\Gd')\right)\left(\sum_{I\in\mathcal P}\norm {f_I }^2_{L^p(\R^2)}\right)^{\frac 1 2},
$$
for some absolute constant $K=K(d)$. Since $f$ and $\mathcal P$ are arbitrary, we have \eqref{hypoeqn}.

\end{proof}
Therefore, the proof of Theorem \ref{poly} is complete.

\section{Appendix}
In the appendix, we prove Proposition 2.1 of \cite{MR4078083} in the special case when $\nu\geq 0$.

\begin{cor}[Proposition 2.1 of \cite{MR4078083}]
Let $\Gd\in 2^{-2\N}$, $d\in [3,\infty)$ and $2\leq p\leq 6$. For each interval $I\sub [0,1]$, let
$$
E_I g(x,y)=\int_I g(s)e(xs+ys^d)ds.
$$
Let $T$ be an axis-parallel rectangle with side lengths $\Gd^{-1}\times \Gd^{-d/2}$ (in the $x$ and $y$ directions, respectively). Let $\eta$ be a nonnegative Schwartz function such that $|\eta|\geq 1$ on $T$ and $\widehat {\eta}$ is supported on the dual rectangle $T^*$ which has side lengths $\Gd\times \Gd^{d/2}$.

Denote $\Delta_k=[(k-1)\Gd^{1/2},k\Gd^{1/2}]$, $1\leq k\leq \Gd^{-1/2}$. Then for any $g\in L^1([0,1])$, we have
$$
\norm {E_{[0,1]}g}_{L^p(T)}\lesssim_\Ge \Gd^{-\Ge}\left(\sum_{k=1}^{\Gd^{-1/2}}\norm{E_{\Delta_{k}}g}^2_{L^p(\eta^p)}\right)^{\frac 1 2},
$$
where the implicit constant depends on $p$ and $\Ge$ only.
\end{cor}
\begin{proof}
Split the partition $\mathcal P:=\{\Delta_k:1\leq k\leq \Gd^{-1/2}\}$ into sub-partitions $\mathcal P_n$, $1\leq n\leq \log_2(\Gd^{-1/2})$, defined as
$$
\mathcal P_n:=\{\Delta_k:2^{n-1}\leq k< 2^n\}.
$$
Denote $I_n$ as the union of intervals in $\mathcal P_n$. Since we can afford logarithmic losses, by triangle inequality and Cauchy-Schwarz it suffices to prove that for each $n$, we have
\begin{equation}\label{small1}
    \norm {E_{I_n}g}_{L^p(T)}\lesssim_\Ge \Gd^{-\Ge}\left(\sum_{\Delta_k\in \mathcal P_n}\norm{E_{\Delta_{k}}g}^2_{L^p(\eta^p)}\right)^{\frac 1 2}.
\end{equation}
Note that for each $n$, $\mathcal P_n$ is a sub-admissible partition of $I_n$ for $s^d$ at scale $a_n:=2^{dn}\Gd^{d/2}$. Therefore, by Theorem \ref{uniform} applied to the monomial $s^d\in \mathcal D_{ d -2}$, for any $f\in L^p(\R^2)$ with Fourier support in $\mathcal N^{s^d}_{I_n,Ca_n}$ (where $C=C_d$ is an absolute constant), we have
\begin{equation}\label{small}
   \norm {f}_{L^p(\R^2)}\lesssim_\Ge \Gd^{-\Ge}\left(\sum_{\Delta_k\in \mathcal P_n} \norm{f_{\Delta_k}}^2_{L^p(\R^2)}\right)^{\frac 1 2}. 
\end{equation}
The proof that \eqref{small} implies \eqref{small1} is purely technical and routine. Such idea can be found in more detail in \cite{Li}, but we include it here for completeness. The idea is to use a standard mollification technique that passes a global neighbourhood version of a decoupling inequality to its local extension version, with the choice of appropriate scales.

Let $I'_n=\cup_{k=2^{n-1}+1}^{2^n-2}\Delta_k$ be the union of intervals in $\mathcal P_k$ excluding the first and the last one. By triangle inequality and Cauchy-Schwarz, it suffices to show
\begin{equation}\label{04}
   \norm {E_{I'_n}g}_{L^p(T)}\lesssim_\Ge \Gd^{-\Ge}\left(\sum_{k=2^{n-1}}^{2^n-1}\norm{E_{\Delta_k}g}^2_{L^p(\eta^p)}\right)^{\frac 1 2}. 
\end{equation}
On the left hand side of \eqref{04}, we have
$$
\norm {E_{I'_n}g}_{L^p(T)}\leq \norm {\eta E_{I'_n}g}_{L^p(\R^2)}.
$$
Observe that $\widehat{E_{I'_n}g}*\widehat{\eta}$ is a smooth function supported on the Minkowski sum of $\Gamma_n$ and $T^*$ where $\Gamma_n$ is the graph of $s^d$ over $I'_n$. Since each $\Delta_k$ has length $\Gd^{1/2}$ and $T^*$ has length $\Gd$ in the horizontal direction, we see that $\Gamma_n+T^*$ does not exceed $[0,1]$ in the horizontal axis. In addition, for each $2^{n-1}\leq k\leq 2^n-1$, $2^{n-2}\leq k'\leq 2^n-1$, we have $(\eta E_{\Delta_{k'}} g)_{\Delta_k}\neq 0$ only if $k'=k-1$, $k$ or $k+1$. 

In the vertical axis, for $(s,s^d)\in \Gamma_n$ and $(u,v)\in T^*$, we have
\begin{align*}
    |s^d+v-(s+u)^d|
    &\leq |v|+|u|\sum_{j=0}^{d-1}(s+u)^j s^{d-1-j}\\
    &\lesssim_d \Gd^{\frac d 2}+\Gd \cdot 2^{n(d-1)}\Gd^{\frac {d-1} 2}\\
    &\lesssim_d 2^{dn}\Gd^{\frac d 2},
\end{align*}
since $2^n\leq \Gd^{-1/2}$ and for $s\in I'_n$ we have $2^{n-1}\Gd^{1/2}\leq s\leq 2^n \Gd^{1/2}$. Thus we have
\begin{equation*}
\Gamma_n+T^*\sub \mathcal N^{s^3}_{I_n,Ca_n}.
\end{equation*}

Thus, for $2^{n-1}\leq k\leq 2^n-1$,
\begin{align*}
    \norm {(\eta E_{I_n'}g)_{\Delta_k}}_{L^p(\R^2)}
    &\leq \norm {(\eta E_{\Delta_{k-1}}g)_{\Delta_k}}_{L^p(\R^2)}+\norm{(\eta E_{\Delta_k}g)_{\Delta_k}}_{L^p(\R^2)}+\norm{(\eta E_{\Delta_{k+1}}g)_{\Delta_k}}_{L^p(\R^2)}\\
    &\lesssim\norm {\eta E_{\Delta_{k-1}}g}_{L^p(\R^2)}+\norm{\eta E_{\Delta_k}g}_{L^p(\R^2)}+\norm{\eta E_{\Delta_{k+1}}g}_{L^p(\R^2)},
\end{align*}
since the Fourier multiplier $(s,t)\mapsto 1_I(s)$ for any interval $I$ has $L^p(\R^2)\to L^p(\R^2)$ operator norm bounded by an absolute constant, whenever $1<p<\infty$.

Thus, by \eqref{small} with $f=\eta E_{I'_n}g$,
\begin{align*}
    \norm {\eta E_{I'_n}g}_{L^p(\R^2)}
    &\lesssim_\Ge \Gd^{-\Ge}\left(\sum_{k=2^{n-1}}^{2^n-1}\norm{ (\eta E_{I'_n}g)_{\Delta_k}}^2_{L^p(\R^2)}\right)^{\frac 1 2}\\
    &\lesssim \Gd^{-\Ge}\left(\sum_{k=2^{n-1}}^{2^n-1}\left(\norm {\eta E_{\Delta_{k-1}}g}_{L^p(\R^2)}+\norm{\eta E_{\Delta_k}g}_{L^p(\R^2)}+\norm{\eta E_{\Delta_{k+1}}g}_{L^p(\R^2)}\right)^2\right)^{\frac 1 2}\\
    &\lesssim \Gd^{-\Ge}\left(\sum_{k=2^{n-1}}^{2^n-1}\norm{\eta E_{\Delta_k}g}^2_{L^p(\R^2)}\right)^{\frac 1 2},
\end{align*}
where in the last line we have used Cauchy-Schwarz.
\end{proof}

\end{document}